\documentclass[12pt]{amsart}

\usepackage{amsmath,amssymb,mathrsfs,graphicx,overpic}

\newtheorem{lemma}{Lemma}[section]
\newtheorem{proposition}[lemma]{Proposition}
\newtheorem{theorem}[lemma]{Theorem}

\theoremstyle{definition}
\newtheorem{definition}[lemma]{Definition}

\theoremstyle{plain}

\def\C{\mathbb C}

\def\Z{\mathbb{Z}}
\def\to{\rightarrow}
\def\SL{\mathrm{SL}}
\def\tr{\mathrm{tr}}
\def\S{\mathbb{S}}

\def\l{\langle\!\langle}
\def\r{\rangle\!\rangle}
\def\Im{\mathrm{Im}}

\def\i{\mathbf{i}}

\title[The ghost character of the (4,5)-torus knot]
{The ghost character of the (4,5)-torus knot and its applications}

\author{Fumikazu Nagasato and Shinnosuke Suzuki}

\address{Department of Mathematics, Meijo University, 
Tempaku, Nagoya 468-8502, Japan}
\email{fukky@meijo-u.ac.jp}

\address{Toyota Technical Development Corporation (TTDC), 
1-9 Imae, Hanamoto-cho, Toyota, Aichi 470-0334, Japan}
\email{}



\begin{document}

\begin{abstract}
We show that the (4,5)-torus knot $T_{4,5}$ admits exactly one ghost character. 
We then show that this ghost character provides the following two important results.  
\begin{enumerate}
\item It is known that for any knot $K$ every (meridionally) trace-free 
$\SL_2(\C)$-representation 
of the knot group $G(K)$ yields an $\SL_2(\C)$-representation of the fundamental group 
$\pi_1(\Sigma_2K)$ 
of the 2-fold branched cover $\Sigma_2K$ of the 3-sphere along $K$. 
This correspondence often but not always provides all $\SL_2(\C)$-representations 
of $\pi_1(\Sigma_2K)$. 
We show by using the ghost character that $T_{4,5}$ is the simplest torus knot
such that $\pi_1(\Sigma_2T_{4,5})$ admits an $\SL_2(\C)$-representation which cannot be 
realized by any trace-free $\SL_2(\C)$-representations.  
\item We show that $T_{4,5}$ is the simplest torus knot that provides 
a counterexample to Ng's conjecture, concerned with 
a polynomial map $h^*$ between the character variety $X(\Sigma_2K)$ of $\pi_1(\Sigma_2K)$ 
and the fundamental variety $F_2(K)$. 
More precisely, the map $h^*$ is surjective but not injective, and hence not an isomorphism 
for $T_{4,5}$. 
\end{enumerate}
\end{abstract}

\maketitle

\section{Background}\label{intro}

In \cite{Nagasato-Yamaguchi}, we discovered several important properties 
of the cross-section of the character variety $X(K)$ of the knot group $G(K)$, 
defined by intersecting $X(K)$ with hyperplane given by the trace-free condition   
$\tr(\rho(\mu_K))=0$, where $\rho$  is a representation $\rho: G(K) \to \SL_2(\C)$, 
and $\mu_K$ is a meridian of $K$.
We refer to this cross-section as the trace-free slice of the character variety $X(K)$ 
(or simply the trace-free slice of a knot $K$), and denote it by $S_0(K)$. 
One of the most significant properties of $S_0(K)$ is that it provides a framework 
in which the characters of trace-free representations yield a large subset 
of the character variety $X(\Sigma_2K)$ of the fundamental group 
$\pi_1(\Sigma_2K)$, where $\Sigma_2K$ is the 2-fold branched cover $\Sigma_2K$ 
of the 3-sphere $\S^3$ along the knot $K$. 
The framework is precisely described by the map $\widehat{\Phi}: S_0(K) \to X(\Sigma_2K)$,  
as introduced in \cite{Nagasato-Yamaguchi}. 
(For details, see Subsection \ref{subsec_trace-free}; 
see also \cite{Nagasato1,Nagasato-Yamaguchi} etc.) 
The map $\widehat{\Phi}$ is a powerful tool 
for providing $\SL_2(\C)$-representations of $\pi_1(\Sigma_2K)$ 
from trace-free $\SL_2(\C)$-representations of $G(K)$\footnote{Such 
an $\SL_2(\C)$-representation of $\pi_1(\Sigma_2K)$ is refered to as 
a $\tau$-equivalent representation. See \cite{Nagasato-Yamaguchi}.}. 
Indeed, the map $\widehat{\Phi}$ is known to be surjective 
for all 2-bridge knots and pretzel knots 
(see \cite[Theorem 1 and Lemma 23]{Nagasato-Yamaguchi}; 
see also \cite[Theorem 1.3]{Nagasato2}).  
Furthermore, it was shown in \cite[Theorem 4.9 (1)]{Nagasato5} that 
the map $\widehat{\Phi}$ is also surjective for any 3-bridge knot. 
From this perspective, the following natural question arises: 
\begin{center}
{\it Is the map $\widehat{\Phi}$ surjective for any knot?}
\end{center}
The answer is negative. 
In general, it is not easy to find a representation $\rho_*: \pi_1(\Sigma_2K) \to \SL_2(\C)$ 
that cannot arise from any trace-free representation of $G(K)$,  
since we basically need to compute all elements of both $S_0(K)$ and $X(\Sigma_2K)$ 
in order to compare them. 
However, the notion of a ghost character of a knot, introduced in \cite{Nagasato4} 
(see also \cite[Definition 4.7]{Nagasato5}), provides a relatively accessible way 
to detect the existence of such an $\SL_2(\C)$-representation of $\pi_1(\Sigma_2T_K)$. 

In \cite[Theorem 4.9 (1), (2)]{Nagasato5}, we established a criterion using ghost characters 
to determine when the map $\widehat{\Phi}$ fails to be surjective. 
In particular, if a knot admits no ghost characters, then the map $\widehat{\Phi}$ is 
surjective. Later, in \cite{Nagasato6}, applying the above criterion to a ghost character of 
the (5,6)-torus knot, we showed that $\widehat{\Phi}$ is not surjective for that knot. 

In the present paper, we show that an even simpler torus knot than $T_{5,6}$ 
admits a ghost character satisfying the criterion in \cite[Theorem 4.9 (2)]{Nagasato5}, 
and hence $\widehat{\Phi}$ is not surjective in this case either.  
\begin{theorem}\label{thm_main}
For the $(4,5)$-torus knot $T_{4,5}$, the map $\widehat{\Phi}$ is not surjective,  
that is, there exists an $\SL_2(\C)$-representation of $\pi_1(\Sigma_2T_{4,5})$ 
that cannot be obtained, via the map $\widehat{\Phi}$,  
from any trace-free representation of $G(T_{4,5})$. 
\end{theorem}
We remark that the (4,5)-torus knot is the simplest torus knot for which 
$\widehat{\Phi}$ is not surjective. 
Indeed, all torus knots simpler than $T_{4,5}$ 
have bridge index at most 3, and thus admit no ghost characters. 
As mentioned above, for such knots, the map $\widehat{\Phi}$ is surjective.  

In the following sections, we briefly review the trace-free slice $S_0(K)$ 
and the notion of ghost characters of knots to explain the results mentioned before. 
We then show that the $(4,5)$-torus knot $T_{4,5}$ admits a ghost character 
that give rise to an $\SL_2(\C)$-representation of $\pi_1(\Sigma_2 T_{4,5})$, 
which essentially establishes the desired result (see Theorem \ref{thm_nag}).  
Furthermore, we demonstrate that this ghost character also yields 
another counterexample to Ng's conjecture. To be more precise, 
the map $h^*$ from the character variety $X(\Sigma_2T_{4,5})$ of 
the fundamental group $\pi_1(\Sigma_2T_{4,5})$ to 
the fundamental variety $F_2(T_{4,5})$, reviewed in Section \ref{sec_2}, 
is surjective but not injective. 


\section{Ghost characters of a knot}\label{sec_2}
Ghost characters of a knot are defined via the trace-free slice of 
the character variety of the knot group. 
We begin by reviewing the character variety of a knot group and its trace-free slice. 

\subsection{Brief review of the character variety of a finitely presented group}
We refer to \cite{Culler-Shalen} for the definition of character varieties.  
Let $G$ be a finitely presented group generated by 
$n$ elements $g_1,\cdots,g_n$. Given a representation $\rho:G \to \SL_2(\C)$, 
let $\chi_{\rho}$ be the character of $\rho$, which is the function on $G$ 
defined by $\chi_{\rho}(g)=\tr(\rho(g))$ for $g \in G$. 
We denote by $\mathfrak{X}(G)$ the set of the characters 
of $\SL_2(\C)$-representations of $G$. 
As shown in \cite{Culler-Shalen, Fricke, Horowiz, Vogt} 
(see also \cite{Gonzalez-Montesinos}), the $\SL_2(\C)$-trace identity
\[
\tr(AB)=\tr(A)\tr(B)-\tr(AB^{-1})\hspace*{0.5cm}(A,B\in\SL_2(\C))
\]
implies that the trace function $t_g(\rho):=\tr(\rho(g))$, for any $g \in G$, 
can be expressed as a polynomial in the traces: 
\[
\begin{array}{ll}
t_{g_i}(\rho) & (1\leq i\leq n),\\
t_{g_ig_j}(\rho) & (1\leq i<j \leq n),\\ 
t_{g_ig_jg_k}(\rho) & (1 \leq i<j<k \leq n). 
\end{array}
\]
It is known that the image of $\mathfrak{X}(G)$ under the map 
\[
t:\mathfrak{X}(G) \to \C^{n+{n\choose 2}+{n\choose 3}},\ 
t(\chi_{\rho})=\left(t_{g_i}(\chi_{\rho});t_{g_ig_j}(\chi_{\rho});t_{g_ig_jg_k}(\chi_{\rho})\right), 
\]
where we extend the trace function by setting $t_{g}(\chi_{\rho})=t_{g}(\rho)$,  
forms a closed algebraic set in $\C^{n+{n\choose 2}+{n\choose 3}}$.  
This algebraic set is called the character variety of $G$, 
and denoted by $X(G)$. 
Although the character varieties have been computed for many classes of knots, 
it remains difficult to determine their defining polynomials in general. 


\subsection{Trace-free slice of the character variety}\label{subsec_trace-free}
Let $G(K)$ denote the knot group of a knot $K$, and $\mu_K$ a meridian of $K$. 
A representation $\rho: G(K) \to \SL_2(\C)$ is said to be trace-free\footnote{Such 
a representation is also referred to as traceless.} if $\tr(\rho(\mu_K))=0$. 
We also refer to its character $\chi_{\rho}$ as a trace-free character. 
The set $\mathfrak{S}_0(K)$ of trace-free characters forms 
a subset of $\mathfrak{X}(K)=\mathfrak{X}(G(K))$: 
\[
\mathfrak{S}_0(K)=\{\chi_{\rho}\in \mathfrak{X}(K) \mid \chi_{\rho}(\mu_K)=0\}. 
\]
Analogously to the construction of $X(K)$, 
the set $\mathfrak{S}_0(K)$ can be realized as a closed algebraic subset $S_0(K)$ 
of the character variety $X(K)$ via the map $t$. 
We refer to this algebraic set $S_0(K)$ as the trace-free slice of $X(K)$, 
or simply the trace free slice of $K$. 
Given a Wirtinger presentation 
\[
G(K)=\langle m_1,\cdots,m_n \mid r_1,\cdots,r_n \rangle, 
\]
the condition $t_{\mu_K}(\chi_{\rho})=0$ implies that $t_{m_i}(\chi_{\rho})=0$ 
for any $1 \leq i \leq n$. Hence we have 
{\small
\[
S_0(K) = t(\mathfrak{S}_0(K))
\cong \left\{\left.\left(t_{m_im_j}(\chi_{\rho});t_{m_im_jm_k}(\chi_{\rho})\right)
\in \C^{{n\choose 2}+{n\choose 3}} \right| \chi_{\rho}\in\mathfrak{S}_0(K) \right\}. 
\]}
Based on this description, the following theorem provides a powerful tool 
for computing the trace-free slice $S_0(K)$. 
For a diagram $D_K$ of a knot $K$, a triple $(i,j,k)$ with $j<k$ is called a Wirtinger triple 
if $i$th arc $a_i$, $j$th arc $a_j$ and $k$th arc $a_k$ of $D_K$ meet 
at a crossing in such a way that $a_i$ is the overarc and $a_j, a_k$ are the underarcs. 

\begin{theorem}[cf. \cite{Nagasato4}, Theorem 3.2 in \cite{Gonzalez-Montesinos}]
\label{defpoly_S0K}
Let $G(K)=\langle m_1,\cdots,m_n \mid r_1,\cdots,r_n \rangle$ be 
a Wirtinger presentation. Then the trace-free slice $S_0(K)$ is isomorphic to 
the following algebraic set in $\C^{{n \choose 2}+{n \choose 3}}$: 
\[
S_0(K) \cong 
\left\{\left.
(x_{12},\cdots,x_{nn-1}; x_{123},\cdots,x_{n-2,n-1,n})\in\C^{{n \choose 2}+{n \choose 3}}
\right| {\rm (F2)},{\rm (GH)} \right\},
\]
where ${\rm (F2)}$ and ${\rm (GH)}$ are the polynomial relations defined as follows: 
\begin{description}
\item[(F2)] the fundamental relations 
\begin{eqnarray*}
&x_{ak}=x_{ij}x_{ai}-x_{aj},&\\
&(1 \leq a \leq n,\ (i,j,k):\mbox{any Wirtinger triple}),&
\end{eqnarray*}
\item[(GH)] the general hexagon relations 
\begin{eqnarray*}
&x_{i_1 i_2 i_3} \cdot x_{j_1 j_2 j_3}
=\dfrac{1}{2}
\left|
\begin{array}{ccc}
x_{i_1 j_1} & x_{i_1 j_2} & x_{i_1 j_3}\\
x_{i_2 j_1} & x_{i_2 j_2} & x_{i_2 j_3}\\
x_{i_3 j_1} & x_{i_3 j_2} & x_{i_3 j_3}
\end{array}\right|,&\\
&(1 \leq i_1<i_2<i_3 \leq n,\ 1 \leq j_1<j_2<j_3 \leq n).&
\end{eqnarray*}
\end{description}
with $x_{ii}=2$, $x_{ji}=x_{ij}$ and 
$x_{i_{\sigma(1)}i_{\sigma(2)}i_{\sigma(3)}}=\mathrm{sign}(\sigma)x_{i_1i_2i_3}$ 
for any element $\sigma$ of the permutation group $\mathfrak{S}_3$ of degree $3$. 
\end{theorem}

The coordinates $x_{ij}$ and $x_{ijk}$ correspond to 
$-t_{m_im_j}(\chi_{\rho})$ and $-t_{m_im_jm_k}(\chi_{\rho})$, respectively. 
The fundamental relation (F2) is symmetric in the indices $j$ and $k$, 
which correspond to the underarcs, since the fundamental relations (F2) with $a=i$ 
for a Wirtinger triple $(i,j,k)$ implies $x_{ik}=x_{ij}$. 


\subsection{Ghost characters of a knot}\label{subsec_ghost}
Suppose that 
\[
G(K)=\langle m_1,\cdots,m_n \mid r_1,\cdots,r_{n-1} \rangle. 
\]
is a Wirtinger presentation of the knot group $G(K)$. 
We define the fundamental variety $F_2(K) \subset \C^{n \choose 2}$ 
as the set of common solutions to the fundamental relations (F2): 
\[
F_2(K)=\left\{(x_{12},\cdots,x_{n-1,n}) \in \C^{n \choose 2}\ \left|\ 
\begin{array}{c}
x_{ak}=x_{ij}x_{ai}-x_{aj} \mbox{ for any $1 \leq a \leq n$}\\
\mbox{and any Wirtinger triple $(i,j,k)$}
\end{array}
\right.\right\}. 
\]
In general, for any closed algebraic set $V \subset \C^N$, 
the set of regular functions on $V$ forms a ring, 
called the coordinate ring of $V$, and denoted by $\mathbf{C}[V]$. 
It is known that $\mathbf{C}[V]$ is isomorphic to the quotient 
of the polynomial ring $\C[z_1,\cdots,z_n]$ by the ideal of polynomials 
vanishing on $V$.   
In particular, the coordinate ring $\mathbf{C}[F_2(K)]$ of the fundamental 
variety $F_2(K)$ is isomorphic to the nilradical quotient of the complexified 
degree $0$ abelian knot contact homology:
\[
(HC_0^{ab}(K) \otimes \C)/\sqrt{0} \cong \mathbf{C}[F_2(K)]. 
\]
In the current paper, we take this correspondence as the definition of $HC_0^{ab}(K)$, 
rather than using the orignal definition.  
For details on this relationship\footnote{This was originally discovered 
in \cite{Nagasato3}.}, see Proposition 4.2 in \cite {Nagasato5}. 
Thus, $F_2(K)$ is a knot invariant up to biregular equivalence, 
which justifies the notation $F_2(K)$.

Some observations suggest that for any knot $K$ with at most 10 crossings, 
every point of $F_2(K)$ lifts to $S_0(K)$ under the relation (GH). 
However, this property does not hold in general. 

\begin{definition}[Ghost characters of a knot \cite{Nagasato4,Nagasato5}]
A point $(x_{ij}) \in F_2(K)$ that does not satisfy one of (GH) is called 
a ghost character of $K$. 
\end{definition}

As shown in \cite[Theorem 4.8]{Nagasato5}, 
knots with bridge index less than 4 do not admit ghost characters. 
Computer calculations suggest that the $(4,q)$-torus knot $T_{4,q}$ 
with odd $q \geq 5$, 
which have bridge index $4$, appear to admit ghost characters. 
In the next subsection, we explicitely compute 
the ghost character for the knot $T_{4,5}$. 


\subsection{Ghost character of the (4,5)-torus knot}
Let $D$ be the diagram of the knot $T_{4,5}$, shown in Figure \ref{diagram_T45}. 
Assign meridians $m_1,\cdots,m_{15}$ as indicated in Figure \ref{diagram_T45}  
for the Wirtinger presentation of $G(T_{4,5})$ associated with $D$. 
\begin{figure}[htbp]
\[
D=\begin{minipage}{11cm}
\begin{overpic}[width=0.95\linewidth]{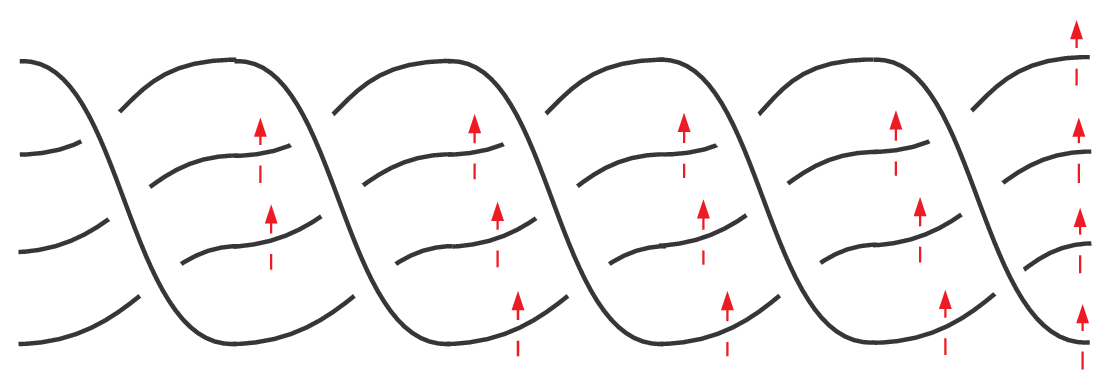}
\put(99,5){$x_1$}
\put(99,14){$x_2$}
\put(99,22.5){$x_3$}
\put(99,31.5){$x_4$}
\put(80,7){$x_5$}
\put(77.5,16){$x_6$}
\put(75,24){$x_7$}
\put(60,7){$x_8$}
\put(58,16){$x_9$}
\put(54.5,24){$x_{10}$}
\put(40,7){$x_{11}$}
\put(38,16){$x_{12}$}
\put(35,24){$x_{13}$}
\put(17,16){$x_{14}$}
\put(16,24){$x_{15}$}
\end{overpic}
\end{minipage}
\]
\caption{Diagram $D$ of $T_{4,5}$ and meridians $m_1,\cdots,m_{15}$. 
The four parallel curves connecting the both sides of the diagram are omitted.}
\label{diagram_T45}
\end{figure}

Under this setup, the fundamental variety $F_2(T_{4,5})$ is described by 
\[
F_2(T_{4,5})=\left\{(x_{12},\cdots,x_{14,15}) \in \C^{15 \choose 2}\ \left|\ 
\begin{array}{c}
x_{ak}=x_{ij}x_{ai}-x_{aj} \mbox{ for all $1 \leq a \leq 15$}\\
\mbox{and all Wirtinger triple $(i,j,k)$}
\end{array}
\right.\right\}
\]
As shown in \cite{Nagasato6}, a knot $K$ in braid position generally admits 
a systematic elimination process for the fundamental relations (F2). 
We apply this method to the current case of $T_{4,5}$, 
following the approach in \cite{Nagasato6}. 
First, for $1 \leq a \leq 15$, we have the following fundamental relations: 
\[
\begin{array}{ll}
x_{a15}=x_{8,11}x_{a11}-x_{a8}, & x_{a14}=x_{11,13}x_{a11}-x_{a13},\\
x_{a13}=x_{5,8}x_{a8}-x_{a5}, & x_{a12}=x_{8,10}x_{a8}-x_{a10},\\
x_{a11}=x_{8,9}x_{a8}-x_{a9}, & x_{a10}=x_{1,5}x_{a5}-x_{a1},\\
x_{a9}=x_{5,7}x_{a5}-x_{a7},  & x_{a8}=x_{5,6}x_{a5}-x_{a6},\\
x_{a7}=x_{1,4}x_{a1}-x_{a4}, & x_{a6}=x_{1,3}x_{a1}-x_{a3},\\
x_{a5}=x_{1,2}x_{a1}-x_{a2}, & \\
x_{a4}=x_{11,12}x_{a11}-x_{a12}, & x_{a3}=x_{4,11}x_{a4}-x_{a11},\\
x_{a2}=x_{4,15}x_{a4}-x_{a15}, & x_{a1}=x_{4,14}x_{a4}-x_{a14}.
\end{array}
\]
Note that the last four types of (F2) are written for triples $(i,j,k)$ with $j>k$,  
utilizing the symmetry between $j$ and $k$, 
for technical reasons related to the elimination process of (F2). 
Although this system can be solved by computer, 
we demonstrate how the elimination can be carried out mostly manually, 
following the strategy in \cite{Nagasato6}.  

To facilitate the process, we introduce the following additional relations derived from (F2): 
\[
\begin{array}{ll}
\fbox{$x_{12}$} &=x_{4,14}x_{24}-x_{2,14}
=x_{4,14}x_{4,15}-(x_{4,15}x_{4,14}-x_{14,15})=\fbox{$x_{14,15}$},\\
\fbox{$x_{13}$} &=x_{4,11}x_{14}-x_{1,11}=x_{4,11}x_{4,14}-(x_{4,14}x_{4,11}-x_{11,14})
=\fbox{$x_{11,14}$},\\
\fbox{$x_{14}$} &=x_{4,14}=\fbox{$x_{11,12}x_{11,14}-x_{12,14}$},\\
\fbox{$x_{23}$} &=x_{4,11}x_{24}-x_{2,11}=x_{4,11}x_{4,15}-(x_{4,15}x_{4,11}-x_{11,15})
=\fbox{$x_{11,15}$},\\
\fbox{$x_{24}$} &=x_{4,15}=\fbox{$x_{11,12}x_{11,15}-x_{12,15}$},\\
\fbox{$x_{34}$} &=x_{4,11}=\fbox{$x_{11,12}$}.
\end{array}
\]
We begin the elimination process by successively eliminating $x_{a15}$ for $1 \leq a \leq 14$, 
using $x_{a15}=x_{8,11}x_{a11}-x_{a8}$.  
Next, we eliminate $x_{a14}$ for $1 \leq a \leq 13$, from both the original (F2) 
and the relations obtained in the previous step, using $x_{a14}=x_{11,13}x_{a11}-x_{a13}$. 
We continue this procedure down to $x_{a5}$. 
At the conclusion of this process, $x_{a15},\cdots,x_{a5}$ for $1 \leq a \leq 15$ 
are expressed as elements in the polynomial ring 
\[
R=\C[x_{12},x_{13},x_{14},x_{23},x_{24},x_{34}].
\] 

To carry out the above elimination process manually, 
we utilize the topological interpretation of (F2), 
as discussed in the proof of Theorem 4.8 in \cite{Nagasato5} 
(see also \cite[Section 2]{Nagasato6}). 
Specifically, (F2) can be interpreted as sliding a corresponding loop 
across a crossing in the diagram $D$ 
(from left to right in the case of the current diagram), 
and resolving the winding by the trace-free Kauffman bracket skein relation: 
\[
\begin{minipage}{10cm}
\begin{overpic}[width=0.95\linewidth]{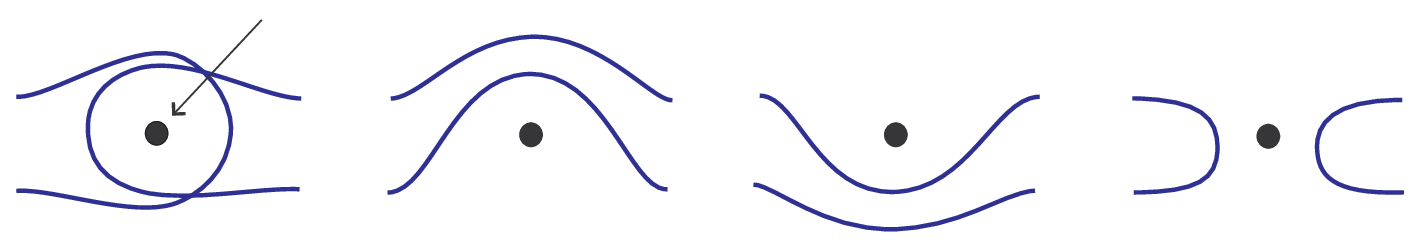}
\put(22.5,7){$=$}
\put(49.5,7){$+$}
\put(75,7){$+$}
\put(19.5,15){\small knot}
\end{overpic}
\end{minipage}.
\]
More precisely, we first identify $x_{ak}$ with a loop freely homotopic to $m_am_k$. 
We regard this loop as the union $c_a \cup c_k$, where $c_a$ and $c_k$ are 
the subarcs of the loop corresponding to $m_a$ and $m_k$, respectively. 
Next, we slide $c_k$ across the crossing, along 
the $k$th arc of $D$, while keeping $c_a$ and the endpoints of $c_k$ fixed. 
We then resolve the winding part (the part of the resulting loop 
that passes under the $i$th arc of $D$) 
using the trace-free Kauffman bracket skein relation: 
\[
\begin{minipage}{4cm}
\begin{overpic}[width=0.95\linewidth]{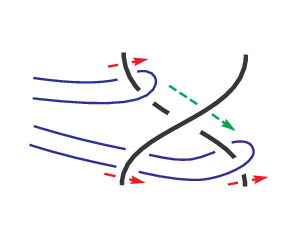}
\put(10,65){$s_{ak}$}
\put(40,10){$i$}
\put(80,10){$j$}
\put(40,70){$k$}
\put(85,48){\small sliding}
{}\end{overpic}
\end{minipage}
=
\begin{minipage}{3.5cm}
\begin{overpic}[width=0.95\linewidth]{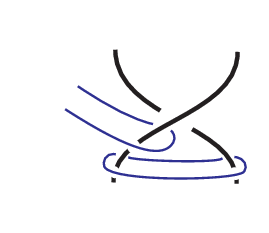}
\put(10,55){$s_{ai}$}
\put(40,8){$i$}
\put(85,8){$j$}
\put(38,72){$k$}
\end{overpic}
\end{minipage}
-
\begin{minipage}{3.5cm}
\begin{overpic}[width=0.95\linewidth]{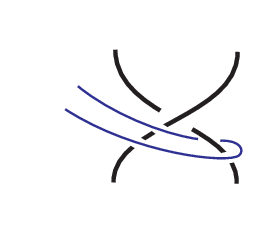}
\put(10,55){$s_{aj}$}
\put(40,8){$i$}
\put(85,8){$j$}
\put(38,72){$k$}
\end{overpic}
\end{minipage}.
\]
Each of the resulting loops is freely homotopic to $m_am_i$, $m_am_j$ or $m_im_j$. 
By interpreting the disjoint union of loops as the product of their associated monomials, 
we recover the relation (F2): $x_{ak}=x_{ij}x_{ai}-x_{aj}$. (This operation can be realized 
via the Kauffman brancket skein algebra. See \cite{Nagasato5}.)
We continue this topological procedure until the given loop is expressed 
as a polynomial in the ring $R$. 

As explained in the proof of \cite[Theorem 4.8]{Nagasato5}, the resulting polynomial 
can also be obtained by first sliding the subarc all the way to the right side of $D$, 
and then resolving the winding parts {\it along the way according to} (F2). 
(For details on this process, see the proof of Theorem 4.8 in \cite{Nagasato5}.)   
We will adopt this perspective in the following arguments as well. 

We now describe the topological elimination process case by case: 
\begin{enumerate}
\item For $s_{ij}$ with $1 \leq i \leq j \leq 4$, 
place the subarc $c_i$ of $s_{ij}$ (resp. $c_j$) on the left side of $D$, 
and slide it all the way to the right side of $D$, while keeping the subarc $c_j$ (resp. $c_i$) 
and the endpoints of $c_i$ fixed. 
Then resolve the winding parts of the resulting loop along the way according to (F2).  
The resulting expression, after substituting $s_{ij}=x_{ij}$, provides a polynomial 
in $R$, denoted by $g_i(x_{ij})$ (resp. $g_j(x_{ij})$)). 

\item For $s_{ij}$ $(1 \leq i < j \leq 4)$ in the additional relations, 
place a loop freely homotopic to $m_im_j$ on the left side of $D$, 
and slide it all the way to the right side of $D$.  
The resulting loop does not have winding parts in the case of $D$. 
 
\item For the other $s_{ij}$ $(i<j)$, 
place a loop freely homotopic to $m_im_j$ in $D$, 
and slide it all the way to the right side of $D$. 
Then resolve the winding parts of the resulting loop along the way acoording to (F2). 
\end{enumerate}
Process (1) yields the relations $x_{ij}=g_i(x_{ij})$ and $x_{ij}=g_j(x_{ij})$ 
for all $1 \leq i < j \leq 4$. 
Process (2) shows that $x_{12}=x_{23}=x_{34}=x_{45}$ and 
$x_{13}=x_{24}$. 
Process (3) provides expressions for $x_{ij}$ (with $5 \leq i$ or $5 \leq j$) 
as polynomials in $R$. 
Note that in Processes (2) and (3), $x_{ii}$ for $1 \leq i \leq 15$ are omitted, 
as they yield only the trivial relation $x_{ii}=2$. 
(The underlying idea originates in \cite{Nagasato3}; 
see also a similar approach in \cite{Ng}.) 

By the above argument, we define a biregular map (an isomorphism onto its image) 
$i: F_2(T_{4,5}) \to \mathrm{Im}(i) \subset \C^{4 \choose 2}$, 
\[
(x_{12},\cdots,x_{14,15}) \mapsto (x_{12},x_{13},x_{14},x_{23},x_{24},x_{34}). 
\]
The resulting equations $x_{ij}=g_i(x_{ij})$ and $x_{ij}=g_j(x_{ij})$ from Process (1) 
become the defining polynomials of $\mathrm{Im}(i)$. Hence, we obtain 
\[
F_2(T_{4,5}) \cong \left\{(x_{12},\cdots,x_{34})\in \C^{4 \choose 2}\ \left|\ 
\begin{array}{l}
\mbox{$x_{ij}=g_{i}(x_{ij})$, $x_{ij}=g_{j}(x_{ij})$ $(1 \leq i \leq j \leq 4)$}\\
\mbox{$x_{12}=x_{23}=x_{34}=x_{14}$, $x_{13}=x_{24}$}
\end{array}\right.\right\}. 
\]

We now can eliminate $x_{14},x_{23},x_{24}$ and $x_{34}$ using the relations 
$x_{12}=x_{23}=x_{34}=x_{14}$ and $x_{13}=x_{24}$. 
The topological process above shows that substituting 
\[
x_{12}=x_{23},\ x_{23}=x_{34},\ x_{34}=x_{14},\ x_{14}=x_{12},\ x_{13}=x_{24},\ x_{24}=x_{13}
\]
into the relations $x_{ij}=g_{i}(x_{ij})$ and $x_{ij}=g_{j}(x_{ij})$ yields 
$x_{i+1,j+1}=g_{i+1}(x_{i+1,j+1})$ and $x_{i+1,j+1}=g_{j+1}(x_{i+1,j+1})$, 
where the indices are taken cyclically from $1$ to $4$, that is, 
if $i$ (resp. $j$) is $4$, then $i+1=1$ (resp. $j+1=1$). 
Hence, the relations $x_{ij}=g_{i}(x_{ij})$, $x_{ij}=g_{j}(x_{ij})$ 
for $(i,j)=(1,4),(2,3),(3,4)$ are reduced to $x_{12}=g_{1}(x_{12})$, $x_{12}=g_{2}(x_{12})$, 
and those for $(i,j)=(2,4)$ are reduced to 
$x_{13}=g_{1}(x_{13})$, $x_{13}=g_{3}(x_{13})$ under this substitution. 
Furthermore, the relation $x_{ii}=g_{i}(x_{ii})$ for $2 \leq i \leq 4$ is reduced to 
$x_{11}=g_{1}(x_{11})$. Therefore, $F_2(T_{4,5})$ is isomorphic to 
the following algebraic set: 
\[
F_2(T_{4,5}) \cong \left\{(x_{12},x_{13}) \in \C^2\ \left|\ 
\begin{array}{l}
x_{1j}=\widetilde{g_{1}}(x_{1j}),\ x_{1j}=\widetilde{g_j}(x_{1j})\ (2 \leq j \leq 3)\\
x_{11}=\widetilde{g_1}(x_{11})
\end{array}\right.\right\}, 
\]
where $\widetilde{g_{i}}(x_{ij})$ (resp. $\widetilde{g_j}(x_{ij})$) denotes the polynomial 
obtained by substituting 
\[
x_{14}=x_{12},\ x_{23}=x_{12},\ x_{24}=x_{13},\ x_{34}=x_{12}
\]
into $g_{i}(x_{ij})$ (resp. $g_j(x_{ij})$). 
Consequently, we obtain the following explicit expressions for 
the defining polynomials of $F_2(T_{4,5})$. 
Let $a=x_{12}$, $b=x_{13}$. Then we have: 
\begin{eqnarray*}
x_{11}=2=\widetilde{g_{1}}(x_{11}) &=& a^5-4a^3b+3a^3+3ab^2-2ab-3a,\\
a=\widetilde{g_{1}}(x_{12}) &=& a^6-4a^4b+2a^4+3a^2b^2+a^2b-5a^2-b^2+2,\\
a=\widetilde{g_{2}}(x_{12}) &=& a^4b-a^4-3a^2b^2+4a^2b+b^3-3b,\\
b=\widetilde{g_{1}}(x_{13}) &=& a^5b-a^5-4a^3b^2+6a^3b+3ab^3-a^3-3ab^2-5ab+3a,\\
b=\widetilde{g_{3}}(x_{13}) &=& a^5-3a^3b+a^3+ab^2+2ab-3a. 
\end{eqnarray*}
Solving these system of equations, we obtain six points $(x_{12},x_{13})=(2,2)$, 
$(-1,1)$, $(Root(z^2-3z+1),-2+2Root(z^2-3z+1))$, $(Root(z^2+z-1),2)$. 
These points constitute the fundamental variety $F_2(T_{4,5})$. 

We now focus on the point $(x_{12},x_{13})=(-1,1) \in F_2(T_{4,5})$. 
From this, we obtain   
\[
(x_{12},x_{13},x_{14},x_{23},x_{24},x_{34})=(-1,1,-1,-1,1,-1). 
\]
In fact, this point does not satisfy the relation (R): 
\[
\left|
\begin{array}{cccc}
2 & x_{12} & x_{13} & x_{14}\\
x_{21} & 2 & x_{23} & x_{24}\\
x_{31} & x_{32} & 2 & x_{34}\\
x_{41} & x_{42} & x_{43} & 2
\end{array}
\right|=0.
\]
We note that any point satisfying (GH) also satisfies the relations (R). 
(See Proposition 3.1 in \cite{Nagasato5}.)  
Therefore, the point $(x_{12},x_{13})=(-1,1)$ does not satisfies at least one of (GH), 
and thus it is a ghost character of $T_{4,5}$. 
It can be verified that all other points in $F_2(T_{4,5})$ satisfy (GH). 
Hence, we obtain the following result.  
\begin{proposition}
The $(4,5)$-torus knot $T_{4,5}$ admits exactly one ghost character
given by $(x_{12},x_{13})=(-1,1) \in F_2(T_{4,5})$. 
\end{proposition}
This proposition, combined with Theorem 4.10 in \cite{Nagasato5}, 
shows that the knot $T_{4,5}$ provides a counterexample 
to Ng's conjecture (see \cite[Conjectures 4.3, 4.4]{Nagasato5}), 
since the map $h^*$ is either not surjective or not injective. 
We will examine this result in more detail in the following sections. 


\section{Proof of Theorem \ref{thm_main}}
We now prove Theorem \ref{thm_main} using ghost characters. 
To begin, we review the polynomial map $\widehat{\Phi}:S_0(K) \to X(\Sigma_2K)$ 
in a general setting.

\subsection{The map $\widehat{\Phi}$}
We start with briefly reviewing the map 
$\widehat{\Phi}:\mathfrak{S}_0(K) \to \mathfrak{X}(\Sigma_2K)$ 
as introduced in \cite{Nagasato-Yamaguchi}. 
Let $p: C_2K \to E_K$ be the 2-fold cyclic cover of the knot exterior $E_K$, 
$\mu_K$ a meridian of $K$, and $\mu_2$ a meridian of $C_2K$ 
whose image under $p$ is homotopic to $\mu_K^2$. 
In this setting, $\pi_1(\Sigma_2K)$ is isomorphic to 
$\pi_1(C_2K)/\l \mu_2 \r$, where $\l \mu_2 \r$ denotes the normal closure 
of the subgroup generated by $\mu_2$. 
The map $p$ induces a natural injection $p_*: \pi_1(C_2K) \to G(K)$, 
so $\pi_1(C_2K) \cong \Im(p_*)$. Since $p_*(\mu_2)=\mu_K^2$, 
it follows that $\pi_1(\Sigma_2K) \cong \Im(p_*)/\l \mu_K^2 \r$, 
Then the map $\widehat{\Phi}:\mathfrak{S}_0(K) \to \mathfrak{X}(\Sigma_2K)$ 
is defined as follows: for $\chi_{\rho} \in \mathfrak{S}_0(K)$ and $g \in \pi_1(\Sigma_2K)$, 
\begin{eqnarray}\label{map_phi}
\widehat{\Phi}(\chi_{\rho})(g)=(\sqrt{-1})^{\alpha(p_*(g))}\chi_{\rho}(p_*(g)), 
\end{eqnarray}
where $\alpha: G(K) \to H_1(E_K)=\langle \mu_K \rangle \cong \Z$ denotes 
the abelianization map. 
(For further details, see \cite{Nagasato-Yamaguchi}.)

Now, following \cite[Section 4]{Nagasato5}, 
we describe the map $\widehat{\Phi}$ as a polynomial map 
$\widehat{\Phi}: S_0(K) \to X(\Sigma_2K)$, using the following theorem. 

\begin{theorem}[\cite{Fox}, cf. \cite{Kinoshita}]\label{thm_Fox}
For the knot group $G(K)=\langle m_1,\cdots,m_n \mid r_1,\cdots,r_{n-1} \rangle$ 
generated by $n$ meridians\footnote{This presentation 
is not necessarily a Wirtinger presentation.}  $m_1,\cdots, m_n$, we have 
\[
\pi_1(\Sigma_2K) \cong \langle m_1m_i\ (1 \leq i \leq n) \mid w(r_j),w(m_1r_jm_1^{-1})\ 
(1 \leq j \leq n-1),  m_i^2\ (1 \leq i \leq n)\rangle,
\] 
where $w(r_j)$ (resp. $w(m_1r_jm_1^{-1})$) denotes the word obtained 
by rewriting $r_j$ (resp. $m_1r_jm_1^{-1}$) in terms of the generators $m_1m_i$. 
\end{theorem}
This theorem is based on a result of Fox\footnote{Fox's theorem 
can be applied to a presentation of a knot group generated by meridians.} in \cite{Fox},
which states that for the presentation of $G(K)$ in Theorem \ref{thm_Fox}, 
and for the injection $p_*: \pi_1(C_2K) \to G(K)$ satisfying $p_*(\mu_2)=m_1^2$, 
we obtain  
\[
\Im(p_*) \cong \langle m_1m_i, m_im_1^{-1} (1 \leq i \leq n) 
\mid w(r_j),w(m_1r_jm_1^{-1})\ (1 \leq j \leq n-1)\rangle 
\]
via the coset decomposition using a Schreier system $\{1,m_1\}$ 
(cf. \cite{Kinoshita}, \cite[Theorem 4.5]{Nagasato5}). 
This presentation gives a concrete description of the character variety $X(\Sigma_2K)$ 
of $\pi_1(\Sigma_2K)$ via the trace function $t$ introduced in Section \ref{intro}. 
For simplicity, we define the following: 
\begin{eqnarray*}
y_{a}(\chi_{\rho_*}) &:=&t_{m_1m_a}(\chi_{\rho_*}),\\
y_{ab}(\chi_{\rho_*}) &:=&t_{(m_1m_a)(m_1m_b)}(\chi_{\rho_*}),\\
y_{abc}(\chi_{\rho_*}) &:=&t_{(m_1m_a)(m_1m_b)(m_1m_c)}(\chi_{\rho_*}),
\end{eqnarray*}
where $\rho_*:\pi_1(\Sigma_2K) \to \SL_2(\C)$ denotes an unspecified representation. 
Then, as discussed in Section \ref{sec_2}, it follows from \cite{Gonzalez-Montesinos} that 
\[
X(\Sigma_2K)=\left\{\left(y_{a}(\chi_{\rho_*}); y_{bc}(\chi_{\rho_*}); y_{def}(\chi_{\rho_*})\right) 
\in \C^{n-1+{n-1 \choose 2}+{n-1 \choose 3}}\ 
\left|\ 
\begin{array}{l}
\chi_{\rho_*} \in \mathfrak{X}(\Sigma_2K)\\
2 \leq a \leq n\\
2 \leq b < c \leq n \\
2 \leq d < e < f \leq n
\end{array}
\right.\right\}. 
\]
By setting 
\[
z_{ab}(\chi_{\rho_*}):=t_{m_am_b}(\chi_{\rho_*})=t_{(m_1m_a)^{-1}(m_1m_b)}(\chi_{\rho_*})
=y_a(\chi_{\rho_*})y_b(\chi_{\rho_*})-y_{ab}(\chi_{\rho_*}),
\] 
for $\chi_{\rho_*} \in \mathfrak{X}(\Sigma_2K)$, we can convert this parametrization to 
\[
X(\Sigma_2K) \cong 
\left\{\left(z_{ab}(\chi_{\rho_*}); y_{def}(\chi_{\rho_*}) \right) \in \C^{{n \choose 2}+{n-1 \choose 3}}\ 
\left|\ 
\begin{array}{ll}
\chi_{\rho_*} \in \mathfrak{X}(\Sigma_2K), & 1 \leq a < b \leq n\\
& 2 \leq d < e < f \leq n
\end{array}
\right.\right\}. 
\]
Using this formulation together with (\ref{map_phi}), we can express the map $\widehat{\Phi}$ 
as the following polynomial map. For any trace-free character $\chi_{\rho}=(x_{ij};x_{ijk})\in S_0(K)$, 
we have 
\begin{eqnarray*}
\widehat{\Phi}((x_{ij};x_{ijk}))
&=&\left(t_{m_am_b}(\widehat{\Phi}(\chi_{\rho}));
t_{(m_1m_d)(m_1m_e)(m_1m_f)}(\widehat{\Phi}(\chi_{\rho}))
\right)
\label{poly_phi}\\
&=&\left(x_{ab}; x_{1d}x_{1e}x_{1f}-\frac{1}{2}(x_{1d}x_{ef}+x_{1e}x_{df}+x_{1f}x_{de})
\right),\nonumber
\end{eqnarray*}
since for any trace-free character $\chi_{\rho} \in S_0(K)$,  
the $\SL_2(\C)$-trace identity (alternatively, the skein relation) 
together with the trace-free condition yields: 
\begin{eqnarray*}
t_{m_a m_b}(\widehat{\Phi}(\chi_{\rho}))&=&t_{m_a m_b}(-\chi_{\rho})
=-t_{m_am_b}(\chi_{\rho})\\
t_{(m_1m_d)(m_1m_e)(m_1m_f)}(\widehat{\Phi}(\chi_{\rho}))
&=&-t_{(m_1m_d)(m_1m_e)(m_1m_f)}(\chi_{\rho})\\
&=&-t_{(m_1m_d)}(\chi_{\rho})t_{(m_1m_e)}(\chi_{\rho})t_{(m_1m_f)}(\chi_{\rho})\\
&&-\frac{1}{2}t_{(m_1m_d)}(\chi_{\rho})t_{(m_em_f)}(\chi_{\rho})\\
&&-\frac{1}{2}t_{(m_1m_e)}(\chi_{\rho})t_{(m_dm_f)}(\chi_{\rho})\\
&&-\frac{1}{2}t_{(m_1m_f)}(\chi_{\rho})t_{(m_dm_e)}(\chi_{\rho}). 
\end{eqnarray*}


\subsection{Ghost characters and surjectivity of $\widehat{\Phi}$}
In the previous work \cite{Nagasato5}, it was shown that a ghost character satisfying 
a certain condition (described below) provides an obstruction 
to the surjectivity of the map $\widehat{\Phi}$. 

\begin{theorem}[Theorem 4.9 (2) in \cite{Nagasato5}]\label{thm_nag}
Let $K$ be a knot with a Wirtinger presentation 
$G(K)=\langle m_1, \cdots,m_n \mid r_1, \cdots, r_n \rangle$. 
Then the map $\widehat{\Phi}$ is not surjective if there exists a representation 
$\rho_*: \Sigma_2(K) \to \SL_2(\C)$ such that  the point $(t_{m_im_j}(\chi_{\rho_*})) \in F_2(K)$
is a ghost character of $K$, where $m_im_j$ for $1 \leq  i < j \leq n$ generate 
$\pi_1(\Sigma_2K)$. 
\end{theorem}

We briefly review the proof of Theorem \ref{thm_nag}. (For details, see \cite{Nagasato5}.)

\begin{proof}
We first note that for any representation $\rho_*: \Sigma_2(K) \to \SL_2(\C)$, 
the tuple $(t_{m_im_j}(\chi_{\rho_*}))$ defines a point in $F_2(K)$. 
Indeed, for any Wirtinger triple $(i,j,k)$ and any $1 \leq a \leq n$, the relations 
$m_am_k=(m_am_i)(m_jm_i)$ and $m_a^2=1$ in $\pi_1(\Sigma_2K) \cong \Im(p_*)/\l m_1^2 \r$, 
imply that 
\[
t_{m_am_k}(\chi_{\rho_*})=t_{m_im_j}(\chi_{\rho_*})t_{m_am_i}(\chi_{\rho_*})-t_{m_am_j}(\chi_{\rho_*}), 
\]
for any character $\chi_{\rho_*} \in X(\Sigma_2K)$. 
This shows that $t_{m_im_j}(\rho_*)$ for $1 \leq i<j \leq n$ satisfy the fundamental relations (F2) 
for $G(K)$, and thus $(t_{m_im_j}(\chi_{\rho_*}))$ defines a point of $F_2(K)$ 
for any representation $\rho_*: \Sigma_2(K) \to \SL_2(\C)$. 
This naturally induces a map $h^*: X(\Sigma_2K) \to F_2(K)$. 
(For more details on $h^*$, see \cite[Subsection 4.3]{Nagasato5}. 
The dual map was introduced in \cite{Ng2}.) 

Now, suppose that there exists a representation 
$\rho_*: \pi_1(\Sigma_2(K)) \to \SL_2(\C)$ such that 
the point $h^*(\chi_{\rho_*})=(t_{m_im_j}(\chi_{\rho_*})) \in F_2(K)$ is a ghost character of $K$. 
Then $\chi_{\rho_*}$ must lie outside $\Im(\widehat{\Phi})$.  
Otherwise, the preimage $\widehat{\Phi}^{-1}(\chi_{\rho_*})$ is not empty and 
its projection to $F_2(K)$, which is identified with $(t_{m_im_j}(\chi_{\rho_*}))$ by (\ref{poly_phi}), 
is not a ghost character, a contradiction. 
\end{proof}

We remark that the preimage $(h^*)^{-1}(\mathbf{x})$ of a point $\mathbf{x} \in F_2(K)$ 
consists of at most two distinct characters. In particular, if $\mathbf{x} \in F_2(K)$ is 
not a ghost character, then $(h^*)^{-1}(\mathbf{x})$ 
consists of exactly one point lying in $\mathrm{Im}(\widehat{\Phi})$.  
On the other hand, for a ghost character in $ F_2(K)$, 
the preimage under $h^*$ lies entirely outside $\mathrm{Im}(\widehat{\Phi})$. 
In other words, if $K$ admits a ghost character 
whose preimage under $h^*$ is non-empty, then there exists 
a representation $\rho_*: \Sigma_2(K) \to \SL_2(\C)$ that cannot be obtained, 
via the map $\widehat{\Phi}$, from any trace-free representation of $G(K)$. 
Such representations necessarily lie in the preimage of ghost characters under the map $h^*$. 
Thus, for any knot that does not admit a ghost character, the map $\widehat{\Phi}$ is surjective.  
(See Theorem 4.9 (1), (2) in \cite{Nagasato5}.)

Since, by \cite[Theorem 4.8]{Nagasato5}, 2-bridge and 3-bridge knots admit 
no ghost characters, the map $\widehat{\Phi}$ is surjective in these cases. 
(See \cite[Theorem 4.9 (1)]{Nagasato5}; see also 
\cite[Theorem 1, Lemma 23]{Nagasato-Yamaguchi} and \cite[Theorem 1.3]{Nagasato2}.) 
Therefore, in order to find a knot for which the map $\widehat{\Phi}: S_0(K) \to X(\Sigma_K)$
fails to be surjective, we focus on knots with bridge index greater than 3. 
The main subject of this paper, the (4,5)-torus knot $T_{4,5}$, is the simplest example 
among torus knots with bridge index greater than 3. 

\subsection{Proof of Theorem \ref{thm_main}}
The following computations are based on the master's thesis of the second author \cite{Suzuki}. 
As shown in Subsection \ref{subsec_ghost}, the $(4,5)$-torus knot $T_{4,5}$ admits 
a single ghost character $(x_{12},x_{13})=(-1,1) \in F_2(T_{4,5})$. 
Hence, by Theorem \ref{thm_nag}, 
if one can define a representation $\rho: \pi_1(\Sigma_2T_{4,5}) \to \SL_2(\C)$ satisfying 
\[
t_{m_1m_2}(\chi_{\rho})=-1,\ t_{m_1m_3}(\chi_{\rho})=1,
\]
then the map $\widehat{\Phi}$ is not surjective. 
Indeed, such a representation exists, leading to the following result.  
\begin{theorem}[cf. Theorem \ref{thm_main}]\label{thm_main2}
For the $(4,5)$-torus knot $T_{4,5}$, the map $\widehat{\Phi}$ is not surjective. 
That is, there exists an $\SL_2(\C)$-representation of $\pi_1(\Sigma_2T_{4,5})$ 
that cannot be obtained, via the map $\widehat{\Phi}$, from any trace-free 
$\SL_2(\C)$-representations of the knot group $G(T_{4,5})$. 
\end{theorem}

\begin{proof}
First, we compute $\pi_1(\Sigma_2T_{4,5})$ using Theorem \ref{thm_Fox}.
\begin{lemma}[cf. Theorem \ref{thm_Fox}]\label{lem_pi_1s2k}
Let $G(T_{4,5})=\langle m_1,\cdots m_{15} \mid r_1,\cdots, r_{14} \rangle$ 
be the Wirtinger presentation\footnote{The last relator $r_{15}$ is eliminated 
due to deficiency-one condition.} 
associated with the diagram $D$ in Figure $\ref{diagram_T45}$.  
Set $x=m_1m_2, y=m_1m_3$ and $z=m_1m_4$. 
Then the fundamental group $\pi_1(\Sigma_2T_{4,5})$ is presented by 
\[
\pi_1(\Sigma_2T_{4,5}) \cong \langle x,y,z \mid w_i\ (1 \leq i \leq 3),
m_j^2\ (1 \leq j \leq 4 )\rangle, 
\]
where the relators $w_i$ are as follows:
\begin{eqnarray*}
w_1 &=& z^{-1} x^{-1} y z^{-1} x z^{-1} y x^{-1} z^{-1},\\
w_2 &=& z^{-1} x^{-1} y z^{-1} y z^{-1} y x^{-1} z^{-1} x,\\
w_3 &=& z^{-1} x^{-1} y z^{-1} y x^{-1} z^{-1} y.
\end{eqnarray*}
\end{lemma}

We illustrate the computation of $\pi_1(\Sigma_2 T_{4,5})$ in Lemma \ref{lem_pi_1s2k}. 
To begin with, using the relators $r_1,\cdots,r_{14}$ of $G(T_{4,5})$, we obtain 
\begin{eqnarray*}
m_5    &=& m_1 m_2 m_1^{-1},\\
m_6    &=& m_1 m_3 m_1^{-1},\\
m_7    &=& m_1 m_4 m_1^{-1},
\end{eqnarray*}
\begin{eqnarray*}
m_8    &=& m_5 m_6 m_5^{-1} = m_1 m_2 m_3 m_2^{-1} m_1^{-1},\\
m_9    &=& m_5 m_7 m_5^{-1} = m_1 m_2 m_4 m_2^{-1} m_1^{-1},\\
m_{10} &=& m_5 m_1 m_5^{-1} = m_1 m_2 m_1 m_2^{-1} m_1^{-1},\\
m_{11} &=& m_8 m_9 m_8^{-1} = m_1 m_2 m_3 m_4 m_3^{-1} m_2^{-1} m_1^{-1},\\
m_{12} &=& m_8 m_{10} m_8^{-1} = m_1 m_2 m_3 m_1 m_3^{-1} m_2^{-1} m_1^{-1},\\
m_{13} &=& m_8 m_5 m_8^{-1} = m_1 m_2 m_3 m_2 m_3^{-1} m_2^{-1} m_1^{-1},\\
m_{14} &=& m_{11} m_{13} m_{11}^{-1} = m_1 m_2 m_3 m_4 m_2 m_4^{-1} m_3^{-1} m_2^{-1} m_1^{-1},\\
m_{15} &=& m_{11} m_8 m_{11}^{-1} = m_1 m_2 m_3 m_4 m_3 m_4^{-1} m_3^{-1} m_2^{-1} m_1^{-1},\\
m_1    &=& m_4 m_{14} m_4^{-1} = m_4 m_1 m_2 m_3 m_4 m_2 m_4^{-1} m_3^{-1} m_2^{-1} 
m_1^{-1} m_4^{-1},\\
m_2    &=& m_4 m_{15} m_4^{-1} = m_4 m_1 m_2 m_3 m_4 m_3 m_4^{-1} m_3^{-1} m_2^{-1} 
m_1^{-1} m_4^{-1},\\
m_3    &=& m_4 m_{11} m_4^{-1} = m_4 m_1 m_2 m_3 m_4 m_3^{-1} m_2^{-1} m_1^{-1} m_4^{-1}.
\end{eqnarray*}
By Tietze transformations, 
the first 11 relations imply that $G(T_{4,5})$ is generated by $m_1,m_2,m_3,m_4$, 
and then the set of relators of $G(T_{4,5})$ is normally generated 
by the last 3 relations. Thus, we obtain
\[
G(T_{4,5})=\langle m_1,m_2,m_3,m_4 \mid w_1,w_2,w_3 \rangle, 
\]
where $w_1$, $w_2$ and $w_3$ are given by 
\begin{eqnarray*}
w_1 &=& m_4 m_1 m_2 m_3 m_4 m_2 m_4^{-1} m_3^{-1} m_2^{-1} m_1^{-1} m_4^{-1} m_1^{-1},\\
w_2 &=& m_4 m_1 m_2 m_3 m_4 m_3 m_4^{-1} m_3^{-1} m_2^{-1} m_1^{-1} m_4^{-1} m_2^{-1},\\
w_3 &=& m_4 m_1 m_2 m_3 m_4 m_3^{-1} m_2^{-1} m_1^{-1} m_4^{-1} m_3^{-1}.
\end{eqnarray*} 
Consequently, by Theorem \ref{thm_Fox}, we have 
\begin{eqnarray*}
\pi_1(\Sigma_2 T_{4,5})
& \cong &  \langle m_1m_2, m_1m_3, m_1m_4 
\mid w_i (1 \leq i \leq 6), m_j^2\ (1 \leq j \leq 4 )\rangle, 
\end{eqnarray*}
where $w_4 = m_1 w_1 m_1^{-1}, w_5 = m_1 w_2 m_1^{-1}, w_6 = m_1 w_3 m_1^{-1}$. 
Note that the relators $w_j$ ($1 \leq j \leq 6$) are expressed
as words in $m_1m_i$ ($1 \leq i \leq 4$). 
For simplicity, let $x=m_1m_2, y=m_1m_3, z=m_1m_4$. Then the relators become
\begin{eqnarray*}
w_1 &=& z^{-1} x^{-1} y z^{-1} x z^{-1} y x^{-1} z^{-1},\\
w_2 &=& z^{-1} x^{-1} y z^{-1} y z^{-1} y x^{-1} z^{-1} x,\\
w_3 &=& z^{-1} x^{-1} y z^{-1} y x^{-1} z^{-1} y,\\
w_4 &=& z x y^{-1} z x^{-1} z y^{-1} x z,\\
w_5 &=& z x y^{-1} z y^{-1} z y^{-1} x z x^{-1},\\
w_6 &=& z x y^{-1} z y^{-1} x z y^{-1}. 
\end{eqnarray*}
We remark that $w_6=w_1^{-1}$, $w_7=x^{-1}w_2^{-1}x$ and $w_8=y^{-1}w_3y$ hold. 
Thus, by Tietze transformations, the relators $w_6$, $w_7$ and $w_8$ are omitted. 
This shows Lemma \ref{lem_pi_1s2k}.  

To complete the proof of Theorem \ref{thm_main2} (Theorem \ref{thm_main}), 
we show that there exists 
a representation $\rho_*: \pi_1(\Sigma_2T_{4,5}) \to \SL_2(\C)$ satisfying 
\[
\tr(\rho_*(m_1m_2))=\tr(\rho_*(m_1m_4))=-1,\ \tr(\rho_*(m_1m_3))=1.
\] 
Since $\tr(\rho_*(m_1m_2)) \neq 2$, 
the representation $\rho_*$ can be conjugated so that 
\[
(\rho_*(m_1m_2),\rho_*(m_1m_3),\rho_*(m_1m_4))
=\left(
\left(\begin{array}{cc}a & 0\\ 0 & a^{-1} \end{array}\right),
\left(\begin{array}{cc}b & c\\ d & e \end{array}\right),
\left(\begin{array}{cc}f & g\\ h & i \end{array}\right) 
\right) \in \SL_2(\C)^3.
\]
Then, we obtain the following representations  
$\rho_*=\rho_{\alpha_i}: \pi_1(\Sigma_2 T_{4,5}) \to \SL_2(\C)$ for $i=1,2$, 
where $\alpha_i$ is a root of the equation $2\alpha^2+\alpha+2=0$. Let $\i:=\sqrt{-1}$. 
\begin{eqnarray*}
&&(\rho_{\alpha_i}(m_1m_2),\rho_{\alpha_i}(m_1m_3),\rho_{\alpha_i}(m_1m_4))\\
&=&\left(
\left(\begin{array}{cc}e^{\frac{2}{3}\pi \i} & 0\\ 0 & e^{-\frac{2}{3}\pi \i} \end{array}\right),
\left(\begin{array}{cc}-\frac{\i}{\sqrt{3}}e^{\frac{\pi}{3}\i} & -\frac{2}{3}\\
 1 & \frac{\i}{\sqrt{3}}e^{-\frac{\pi}{3}\i} \end{array}\right),
\left(\begin{array}{cc}\frac{\i}{\sqrt{3}}e^{\frac{\pi}{3}\i} & \frac{1+2\alpha_i}{3}\\
\alpha & -\frac{\i}{\sqrt{3}}e^{-\frac{\pi}{3}\i} \end{array}\right) 
\right). 
\end{eqnarray*}
One can easily check that 
\begin{eqnarray*}
&& \tr(\rho_{\alpha_i}(m_1m_2))=2\cos\left(\frac{2}{3}\pi\right)=-1,\\
&& \tr(\rho_{\alpha_i}(m_1m_3))=-\frac{\i}{\sqrt{3}} \cdot 2\i \sin\left(\frac{\pi}{3}\right)=1,\\
&&\tr(\rho_{\alpha_i}(m_1m_4))=\frac{\i}{\sqrt{3}} \cdot 2\i \sin\left(\frac{\pi}{3}\right)=-1. 
\end{eqnarray*}
By Theorem \ref{thm_nag}, this indicates that $\rho_{\alpha_i}$ for $i=1,2$ cannot be obtained 
via $\widehat{\Phi}$ from any trace-free representations of $G(T_{4,5})$, completing the proof. 
\end{proof}

Additionally, we obtained the following representations 
$\rho_*: \pi_1(\Sigma_2T_{4,5}) \to \SL_2(\C)$, corresponding to the remaining points of 
$F_2(T_{4,5})$: 
\[
\begin{array}{lll}
\rho_*(m_1m_2)=
\left(\begin{array}{cc}e^{\frac{2\pi k}{5}\i} & 0\\ 0 & e^{-\frac{2\pi k}{5}\i} \end{array}\right)
& \Rightarrow & 
\tr(\rho_*(m_1m_2))=2\cos\left(\frac{2\pi k}{5}\right)=2\mbox{ or }\frac{-1 \pm \sqrt{5}}{2},\\
\rho_*(m_1m_3)=
\left(\begin{array}{cc}1 & 0\\ 0 & 1 \end{array}\right)
& \Rightarrow & 
\tr(\rho_*(m_1m_3))=2,\\
\rho_*(m_1m_4)=
\left(\begin{array}{cc}e^{\frac{2\pi k}{5}\i} & 0\\ 0 & e^{-\frac{2\pi k}{5}\i} \end{array}\right)
& \Rightarrow &
\tr(\rho_*(m_1m_4))=2\cos\left(\frac{2\pi k}{5}\right)=2\mbox{ or }\frac{-1 \pm \sqrt{5}}{2},\\
\end{array}
\]
where $k=0,1,2,3,4$, 
\[
\begin{array}{lll}
\rho_*(m_1m_2)=
\left(\begin{array}{cc} \frac{3 \pm \sqrt{5}+\beta}{4} & 0\\
0 & \frac{3 \pm \sqrt{5}-\beta}{4} \end{array}\right)
& \Rightarrow & 
\tr(\rho_*(m_1m_2))=\frac{3 \pm \sqrt{5}}{2},\\
\rho_*(m_1m_3)=
\left(\begin{array}{cc}\frac{66(1 \pm \sqrt{5})+26\beta+\beta^3}{132} & \frac{1 \pm 3\sqrt{5}}{11}\\
1 & \frac{66(1 \pm \sqrt{5})-26\beta-\beta^3}{132} \end{array}\right)
& \Rightarrow & 
\tr(\rho_*(m_1m_3))=1 \pm \sqrt{5},\\
\rho_*(m_1m_4)=
\left(\begin{array}{cc}\frac{33(3 \pm \sqrt{5})-7\beta+\beta^3}{132} & \frac{1 \pm 3\sqrt{5}}{11}\\
1 & \frac{33(3 \pm \sqrt{5})+7\beta-\beta^3}{132} \end{array}\right)
& \Rightarrow &
\tr(\rho_*(m_1m_4))=\frac{3 \pm \sqrt{5}}{2}, 
\end{array}
\]
where $\beta=\sqrt{-2\pm 6\sqrt{5}}$. 
These representations verify the surjectivity 
of the map $h^*: X(\Sigma_2 T_{4,5}) \to F_2(T_{4,5})$. 


\section{Another counterexample of Ng's conjecture}\label{Ng_conj}
It has been shown that the (4,5)-torus knot $T_{4,5}$ is the simplest torus knots  
that provides a counterexample to Ng's conjecture, 
since $T_{4,5}$ admits a ghost character, whereas no simpler torus knot does 
(see Theorem 4.10 in \cite{Nagasato5}). We investigate this phenomenon in detail below. 
 
As noted earlier, a representation $\rho_*: \pi_1(\Sigma_2 T_{4,5}) \to \SL_2(\C)$ 
determines a point in $F_2(T_{4,5})$ via the traces: 
\[
(\tr(\rho_*(m_1m_2)), \tr(\rho_*(m_1m_3))).
\] 
In general, this correspondence is formulated by the map  
\[
h^*: X(\Sigma_2K) \to F_2(K), 
\]
defined by $h^*((z_{ab};y_{def}))=(z_{ab})$. 

Ng's conjecture asserts that $h^*$ is an isomorphism for any knot 
(see \cite[Conjectures 4.3, 4.4]{Nagasato5}; see also \cite[Conjecture 5.7]{Ng2}).  
As shown in \cite{Nagasato6}, the (5,6)-torus knot $T_{5,6}$ provides 
the first counterexample to this conjecture. 
More precisely, $T_{5,6}$ admits three ghost characters. 
These ghost characters demonstrate that the map $h^*$ is neither surjective nor injective, 
and hence not an isomorphism. 

In contrast, the (4,5)-torus knot $T_{4,5}$ also admits a single ghost character 
$\mathbf{g}=(-1,1)$. 
However, the preimage of this ghost character $\mathbf{g}$ 
under the map $h^*$ is non-empty, and hence $h^*$ is surjective. 
On the other hand, the computations from the previous section show  
that the preimage $(h^*)^{-1}((-1,1))$ consists of exactly two distinct points. 
Indeed, for the representations $\rho_{\alpha_i}: \pi_1(\Sigma_2T_{4,5}) \to \SL_2(\C)$ 
($i=1,2$) defined earlier, one can verify that the traces
\[
y_{234}(\chi_{\rho_{\alpha_i}})=t_{(m_1m_2)(m_1m_3)(m_1m_4)}(\chi_{\rho_{\alpha_i}}),
\]
are distinct. Hence the characters $\chi_{\rho_{\alpha_1}}$ and $\chi_{\rho_{\alpha_2}}$ 
represent distinct points in $X(\Sigma_2T_{4,5})$. 
This shows that $h^*$ is not injective, and therefore not an isomorphism. 
Consequently, Ng's conjecture does not hold for $T_{4,5}$ either. 
This example provides the first known case in which the map $h^*$ is 
surjective but not injective. 

As shown in Theorem 4.8 in \cite{Nagasato5}, all 2-bridge and 3-bridge knots 
admit no ghost characters. Then, by Theorem 4.10 in \cite{Nagasato5}, 
the map $h^*$ is an isomorphism for such knots. Since any torus knot simpler than $T_{4,5}$ 
has bridge index less than 4, the (4,5)-torus knot is the simplest torus knot with this property. 
 

\section*{Acknowledgements}
The authors would like to thank Yoshikazu Yamaguchi for helpful comments 
on $\tau$-equivalent representations and ghost characters. 
The first author was partially supported during the early stages of this research
by JSPS Research Fellowships for Young Scientists, JSPS KAKENHI for Young Scientists (Start-up), 
MEXT KAKENHI for Young Scientists (B) and JSPS KAKENHI for Young Scientists (B).  
The present work was partially supported by JSPS KAKENHI (C) Grant Number 20K03619.


\end{document}